\documentclass[runningheads,a4paper]{llncs}

\usepackage{amssymb}
\usepackage{graphicx}

\usepackage{multicol}
\usepackage{multirow}

\usepackage{url}
\urldef{\mailsa}\path|{alfred.hofmann, ursula.barth, ingrid.haas, frank.holzwarth,|
\urldef{\mailsb}\path|anna.kramer, leonie.kunz, christine.reiss, nicole.sator,|
\urldef{\mailsc}\path|erika.siebert-cole, peter.strasser, lncs}@springer.com|

\usepackage{amsmath}
\usepackage{amssymb}
\usepackage{graphicx}
\usepackage{color}
\usepackage{ifpdf}
\usepackage{url}
\usepackage{hyperref}
\usepackage{algorithm}
\usepackage{algorithmic}

\newcommand{\Eproof}{\hfill$\square$}

\newcommand{\R}{\mathbb{R}}

\newcommand{\abs}[1]{\left\vert#1\right\vert}
\newcommand{\set}[1]{\left\{#1\right\}}
\newcommand{\norm}[1]{\left\Vert#1\right\Vert}
\newcommand{\iprods}[1]{\langle#1\rangle}
\newcommand{\eofprove}{\hfill $\square$}

\newcommand{\Fsc}{\mathcal{F}_{\mathrm{scl}}}
\newcommand{\prox}{\mathrm{prox}}
\newcommand{\Sc}{\mathcal{S}}
\newcommand{\Pc}{\mathcal{P}}

\newcommand{\Fs}{$\mathcal{F}_{\rm scl} $}

\newcommand{\xb}{\mathbf{x}}
\newcommand{\Db}{\mathbf{D}}
\newcommand{\xopt}{\mathbf{x}^{\star}}
\newcommand{\yb}{\mathbf{y}}
\newcommand{\zb}{\mathbf{z}}
\newcommand{\ub}{\mathbf{u}}
\renewcommand{\sb}{\mathbf{s}}
\newcommand{\vb}{\mathbf{v}}
\newcommand{\db}{\mathbf{d}}
\newcommand{\wb}{\mathbf{w}}

\newcommand{\Hb}{\mathbf{H}}

\newcommand{\Xb}{\mathbf{X}}

\newcommand{\dom}[1]{\mathrm{dom}\left(#1\right)}

\begin{document}

\mainmatter  

\title{\vspace{-2ex}Composite convex minimization involving self-concordant-like cost functions\vspace{-2ex}\footnote{The first version was uploaded on Feb 04, 2015. This is an updated version, which corrects a mistake in Lemma 1.}}
 
\titlerunning{Composite convex minimization involving self-concordant-like cost functions}

%
%
\author{\vspace{-0ex}Quoc Tran-Dinh, Yen-Huan Li 
\textit{\and} Volkan Cevher}
\authorrunning{Quoc Tran-Dinh, Yen-Huan Li \textit{and} Volkan Cevher}

\institute{\vspace{-2ex}Laboratory for Information and Inference Systems (LIONS)\\
 EPFL, Lausanne, Switzerland}
 
\toctitle{Composite convex minimization involving self-concordant-like cost functions}
\tocauthor{Quoc Tran-Dinh and Volkan Cevher}
\maketitle

\vspace{-2ex}
\begin{abstract}
The self-concordant-like property of a smooth convex function is a new analytical structure that generalizes the self-concordant notion.
While a wide variety of important applications feature the self-concordant-like property, this concept has heretofore remained unexploited in convex optimization. 
To this end, we develop a variable metric framework of minimizing the sum of a ``simple'' convex function and a self-concordant-like function.  
We introduce a new analytic step-size selection procedure and prove that the basic gradient algorithm has improved convergence guarantees as compared to ``fast'' algorithms that rely on the Lipschitz gradient property. 
Our numerical tests with real-data sets show that the practice indeed follows the theory. 
\end{abstract}

\vspace{-4ex}
\section{Introduction}\label{sec:intro}
\vspace{-2ex}
In this paper, we consider the following composite convex minimization problem:
\begin{equation}\label{eq:composite_min_Fx}
F^{\star} := \min_{\xb\in\mathbb{R}^n}\set{F(\xb) := f(\xb) + g(\xb)},
\end{equation}
where $f$ is a nonlinear smooth convex function, while $g$ is a ``simple'' possibly nonsmooth convex function.
Such composite convex problems naturally arise in many applications of machine learning, data sciences, and imaging science.
Very often, $f$ measures a data fidelity or a loss function, and $g$ encodes a form of low-dimensionality, such as sparsity or low-rankness. 

To trade-off accuracy and computation optimally in large-scale instances of \eqref{eq:composite_min_Fx}, 
existing optimization methods invariably invoke the additional assumption that the smooth function $f$ also has an $L$-Lipschitz continuous gradient (cf., \cite{Nesterov2007} for the definition).
A highlight is the recent developments on proximal gradient methods, which feature (nearly) dimension-independent, global sublinear convergence rates \cite{Beck2009,Mine1981,Nesterov2007}. 
When the smooth $f$ in  \eqref{eq:composite_min_Fx} also has strong regularity \cite{Robinson1980}, the problem \eqref{eq:composite_min_Fx} is also within the theoretical and practical grasp of proximal-(quasi) Newton algorithms with linear, superlinear, and quadratic convergence rates \cite{Becker2012b,Lee2014,Schmidt2011}. These algorithms specifically exploit second order information or its principled approximations (e.g., via BFGS or L-BFGS updates \cite{Nocedal2006}). 

In this paper, we do away with the Lipschitz gradient assumption and instead focus on another structural assumption on $f$ in developing an algorithmic framework for \eqref{eq:composite_min_Fx}, which is defined below.

\begin{definition}\label{de:self_concordant_type}
A convex function $f\in \mathcal{C}^3(\R^n)$ is called a self-concordant-like function $f \in \mathcal{F}_{\rm scl}$, if: 
\begin{equation}\label{eq:self_concordant_type}
\vert\varphi'''(t)\vert \leq M_f\varphi''(t) \norm{\ub}_2,
\end{equation}
for $t \in \mathbb{R}$ and $M_f >0$, where $\varphi(t) := f(\xb + t\ub)$ for any $\xb\in\mathrm{dom}(f)$ and $\ub \in \mathbb{R}^n$. 
\end{definition}

Definition \ref{de:self_concordant_type} mimics the standard self-concordance concept (\cite[Definition 4.1.1]{Nesterov2004})  and was first discussed in \cite{Bach2009} for model consistency in logistic regression. 
For composite convex minimization, self-concordant-like functions abound in machine learning, including but not limited to logistic regression, multinomial logistic regression, conditional random fields, and robust regression (cf., the references in \cite{Bach2013a}). 
In addition, special instances of geometric programming \cite{Boyd2004} can also be recast as \eqref{eq:composite_min_Fx} where $f \in \mathcal{F}_{\rm scl}$. 

The importance of the assumption $f \in \mathcal{F}_{\rm scl}$ in  \eqref{eq:composite_min_Fx} is twofold. 
First, it enables us to derive an explicit step-size selection strategy for proximal variable metric methods, enhancing backtracking-line search operations with improved theoretical convergence guarantees. For instance, we can prove that our proximal gradient method can automatically adapt to the local strong convexity of $f$ near the optimal solution to feature linear convergence under mild conditions. This theoretical result is backed up by great empirical performance on real-life problems where the fast Lipschitz-based methods actually exhibit sublinear convergence (cf. Section \ref{sec:app}). Second, the self-concordant-like assumption on $f$ also helps us provide scalable numerical solutions of \eqref{eq:composite_min_Fx} for specific problems where $f$ does not have Lipschitz continuous gradient, such as special forms of geometric programming problems. 

\vspace{0.5ex}
\noindent\textbf{Contributions. } Our specific contributions can be summarized as follows:
\begin{enumerate}
\vspace{-1ex}
\item We propose a new \textit{variable metric} framework for minimizing the sum $f+g$ of a self-concordant-like function $f$ and a convex, possibly nonsmooth function $g$. Our approach relies on the solution of a convex subproblem obtained by linearizing and regularizing the first term $f$, and uses an \textit{analytical} step-size to achieve descent in three classes of algorithms: first order methods, second order methods, and quasi-Newton methods. \vspace{-0mm}

  \item We establish both the global and the local convergence of different variable metric strategies. We pay particular attention to diagonal variable metrics since in this case many of the proximal subproblems can be solved exactly. We derive conditions on when and where these variants achieve locally linear convergence. When the variable metric is the Hessian of $f$ at  each iteration, we show that the resulting algorithm locally exhibits quadratic convergence without requiring any globalization strategy such as a backtracking line-search.\vspace{0mm}
 
 \item We apply our algorithms to large-scale real-world  and synthetic problems to highlight the strengths and the weaknesses of our variable-metric scheme. \vspace{-1ex}
\end{enumerate}

\noindent\textbf{Relation to prior work.} 
Many of the composite problems with self-concordant-like $f$, such as regularized logistics and multinomial logistics, also have Lipschitz continuous gradient. In those specific instances, many theoretically efficient algorithms are applicable \cite{Beck2009,Becker2012b,Lee2014,Mine1981,Nesterov2007,Schmidt2011}. 
Compared to these works, our framework has theoretically stronger local  convergence guarantees thanks to the specific step-size strategy matched with $f \in \mathcal{F}_{\rm scl}$. The authors of \cite{Tran-Dinh2013a} consider composite problems where $f$ is standard self-concordant and proposes a proximal Newton algorithm optimally exploiting this structure. Our structural assumptions and algorithmic emphasis here are different.

\vspace{0.25ex}
\noindent\textbf{Paper organization.}
We first introduce the basic definitions and optimality conditions before deriving the variable metric strategy in Section \ref{sec:composite_opt}. 
Section \ref{sec:alg_framework} proposes our new variable metric framework, describes its step-size selection procedure, and establishes the convergence theory of its variants. Section \ref{sec:app} illustrates our framework in real and synthetic data. 
\vspace{-2ex}
\section{Preliminaries}\label{sec:composite_opt}
\vspace{-2ex}
We adopt the notion of self-concordant functions in \cite{Nesterov2004,Nesterov1994} to a different smooth function class. 
Then we present the optimality condition of problem \eqref{eq:composite_min_Fx}.

\vspace{-3ex}
\subsection{Basic definitions}
\vspace{-1.5ex}
Let $g : \R^n \to \R$ be a proper, lower semicontinuous convex function \cite{Rockafellar1970} and $\mathrm{dom}(g)$ denote the domain of $g$. 
We use $\partial{g}(\xb)$ to denote the subdifferential of $g$ at $\xb\in\dom{g}$ if $g$ is nondifferentiable at $\xb$ and $\nabla{g}(\xb)$ to denote its gradient, otherwise.
Let $f: \R^n \to \R$ be a $\mathcal{C}^3(\mathrm{dom}(f))$ function (i.e., $f$ is three times continuously differentiable). 
We denote by $\nabla{f}(\xb)$ and $\nabla^2f(\xb)$ the gradient and the Hessian of $f$ at $\xb$, respectively. 
Suppose that, for a given $\xb\in\dom{f}$, $\nabla^2f(\xb)$ is positive definite (i.e., $\nabla^2f(\xb)\in\Sc^n_{++}$), we define the local norm of a given vector $\ub\in\R^n$ as $\norm{\ub}_{\xb} := [\ub^T\nabla^2f(\xb)\ub]^{1/2}$. 
The corresponding dual norm of $\ub$, $\norm{\ub}^{*}_{\xb}$ is defined as $\norm{\ub}^{*}_{\xb} := \max\set{ \ub^T\vb ~|~ \norm{\vb}_{\xb} \leq 1} = [\ub^T\nabla^2f(\xb)^{-1}\ub]^{1/2}$.   

\vspace{-2ex}
\subsection{Composite self-concordant-like minimization}\label{subsec:min_composite}
\vspace{-1ex}
Let $f\in$\Fs$(\R^n)$ and $g$ be proper, closed and convex. 
The optimality condition for \eqref{eq:composite_min_Fx} can be concisely written as follows:
\begin{equation}\label{eq:optimality2}
0 \in \nabla{f}(\xopt) + \partial{g}(\xopt).
\end{equation}
Let us denote by $\xopt$ as an optimal solution of \eqref{eq:composite_min_Fx}. 
Then, the condition \eqref{eq:optimality2} is necessary and sufficient.
We also say that $\xopt$ is \textit{nonsingular} if $\nabla^2f(\xopt)$ is positive definite.
We now establish the existence and uniqueness of the solution $\xopt$ of \eqref{eq:composite_min_Fx}, whose proof can be found in the appendix.

\vspace{-1ex}
\begin{lemma}\label{le:existence_sol}
Suppose that $f\in$\Fs$(\R^n)$ satisfies Definition \ref{de:self_concordant_type} for some $M_f > 0$. 
Suppose further that $\nabla^2f(\xb) \succ 0$ for some $\xb\in\mathrm{dom}(f)$. 
In addition, $\lambda(\xb) := \Vert \nabla{f}(\xb) + \vb\Vert_{\xb}^{\ast} < \frac{\sqrt{\underline{\sigma}(\xb)}}{M_f}$ for some $\vb\in\partial{g}(\xb)$, where $\underline{\sigma}(\xb) := \lambda_{\min}(\nabla^2{f}(\xb))$, the smallest eigenvalue of $\nabla^2{f}(\xb)$.
Then  the solution $\xopt$ of \eqref{eq:composite_min_Fx} exists and is unique.
\end{lemma}
\vspace{-1ex}
For a given symmetric positive definite matrix $\Hb$, we define a generalized proximal operator $\prox_{\Hb^{-1}g}$ as:
\begin{equation}\label{eq:prox_oper}
\prox_{\Hb^{-1}g}(\xb) := \mathrm{arg}\min_{\zb}\big\{g(\zb) + (1/2)\norm{\zb - \xb}_{\Hb^{-1}}^2\big\}.
 \end{equation}
 Due to the convexity of $g$, this operator is well-defined and single-valued. If we can compute $\prox_{\Hb^{-1}g}$ efficiently (e.g., by a closed form or by polynomial time algorithms), then we say that $g$ is \textit{proximally tractable}. Examples of proximal tractability convex functions can be found, e.g., in \cite{Parikh2013}.
 Using $\prox_{\Hb^{-1}g}$, we can write condition \eqref{eq:composite_min_Fx} as:
\begin{equation*}
\xopt - \Hb^{-1}\nabla{f}(\xopt) \in (\mathbb{I} + \Hb^{-1}\partial{g})(\xopt)  \iff \xopt = \prox_{\Hb^{-1}g}(\xopt - \Hb^{-1}\nabla{f}(\xopt)).
\end{equation*}
This expression shows that $\xopt$  is a fixed point of $\mathcal{R}_{\Hb}(\cdot) := \prox_{\Hb^{-1}g}((\cdot) - \Hb^{-1}\nabla{f}(\cdot))$. 
Based on the fixed point principle, one can expect that the iterative sequence $\set{\xb^k}_{k\geq 0}$ generated by $\xb^{k+1} := \mathcal{R}_{\Hb}(x^k)$  converges to $\xopt$. This observation is made rigorous below.
\vspace{-2ex}
\section{Our variable metric framework}\label{sec:alg_framework}
\vspace{-1.5ex}
We first present a generic variable metric proximal framework for solving \eqref{eq:composite_min_Fx}. 
Then, we specify this framework to obtain three variants: proximal gradient, proximal Newton and proximal quasi-Newton algorithms.

\vspace{-2ex}
\subsection{Generic variable metric proximal algorithmic framework}
\vspace{-1ex}
Given $\xb^k\in\dom{F}$ and an appropriate choice $\Hb_k\in\Sc^n_{++}$, since $f\in\Fsc$, one can approximate $f$ at $\xb^k$ by the following quadratic model:
\begin{equation}\label{eq:Q_model}
Q_{\Hb_k}(\xb,\xb^k) := f(\xb^k) + \iprods{\nabla{f}(\xb^k), \xb - \xb^k} + \frac{1}{2}\iprods{\Hb_k(\xb-\xb^k), \xb - \xb_k}.
\end{equation}
Our algorithmic approach uses the variable metric forward-backward framework to generate a sequence $\set{\xb^k}_{k\geq 0}$ starting from $\xb^0\in\dom{F}$ and update:
\begin{equation}\label{eq:x_k1}
\xb^{k+1} := \xb^k + \alpha_k\db^k
\end{equation}
where $\alpha_k \in (0, 1]$ is a given step-size and $\db^k$ is a search direction defined by:
\begin{equation}\label{eq:subprob1}
\db^k := \sb^k - \xb^k, ~~\text{with}~~\sb^k  := \mathrm{arg}\!\min_{\xb}\set{Q_{\Hb_k}(\xb,\xb^k) + g(\xb)}.
\end{equation} 
In the rest of this section, we explain how to determine the step size $\alpha_k$ in the iterative scheme \eqref{eq:x_k1} optimally for special cases of $\Hb_k$. 
For this, we need the following definitions:
\begin{equation}\label{eq:lbd_k_and_beta_k}
\lambda_k := \Vert\db^k\Vert_{\xb^k}, ~~ r_k := M_f\Vert\db^k\Vert_2, ~~\textrm{and}~~ \beta_k := \Vert\db^k\Vert_{\Hb_k} = \iprods{\Hb_k\db^k,\db^k}^{1/2}.
\end{equation}

\vspace{-2ex}
\subsection{Proximal-gradient algorithm}\label{subsec:proximal_gradient}
\vspace{-1ex}
When the variable matrix $\Hb_k$ is \textit{diagonal} and $g$ is proximally tractable, we can efficiently obtain the solution of the subproblem \eqref{eq:subprob1} in a distributed fashion or even in a closed form. 
Hence, we consider $\Hb_k = \Db_k := \mathrm{diag}(\Db_{k,1},\cdots, \Db_{k,n})$ with $\Db_{k,i} > 0$, for $i=1,\cdots, n$. 
Lemma \ref{le:grad_alg}, whose proof is in the appendix, 
provides a step-size selection procedure and proves the global convergence of this proximal-gradient algorithm.

\vspace{-1ex}
\begin{lemma}\label{le:grad_alg}
Let $\set{\xb^k}_{k\geq 0}$ be a sequence generated by \eqref{eq:x_k1} and \eqref{eq:subprob1} starting from $\xb^0\in\dom{F}$.
For $\lambda_k$, $r_k$ and $\beta_k$ defined by \eqref{eq:lbd_k_and_beta_k}, we consider the step-size $\alpha_k$ as:
\begin{equation}\label{eq:step_size}
\alpha_k :=   \frac{1}{r_k}\ln\left( 1 + \frac{\beta_k^2r_k}{\lambda_k^2}\right),
\end{equation}
If  $\beta_k^2r_k \leq (e^{r_k}-1)\lambda_k^2$, then $\alpha_k \in (0, 1]$ and:
\begin{equation}\label{eq:grad_descent}
F(\xb^{k+1}) \leq F(\xb^k) -  \frac{\beta_k^2}{r_k}\left[\left(1 + \frac{\lambda^2_k}{r_k\beta^2_k}\right)\ln\left(1 + \frac{\beta^2_kr_k}{\lambda^2_k}\right) - 1\right].
\end{equation}
Moreover, this step-size $\alpha_k$ is optimal $($w.r.t. the worst-case performance$)$.
\end{lemma}
\vspace{-1ex}

By our condition, the second term on the right-hand side of \eqref{eq:grad_descent} is always positive, establishing that the sequence $\set{F(\xb^k)}$ is decreasing. 
Moreover, as $e^{r_k} - 1 \geq r_k$, the condition $\beta_k^2r_k \leq (e^{r_k}-1)\lambda_k^2$ can be simplified to $\beta_k \leq \lambda_k$. It is easy to verify that this is satisfied whenever $\Db_k\preceq \nabla^2f(\xb^k)$. 
In such cases, our step-size selection ensures the best decrease of the objective value regarding the self-concordant-like structure of $f$ (and not the actual objective instance). 
When  $\beta_k > \lambda_k$, we scale down $\Db_k$ until $\beta_k \leq \lambda_k$. 
It is easy to prove that the number of backtracking steps to find $\Db_{k,i}$ is time constant. 

Now, by using our step-size \eqref{eq:step_size}, we can describe the proximal-gradient algorithm as in Algorithm \ref{alg:A1}.
\begin{algorithm}[!ht]\caption{(Proximal-gradient algorithm with a \textit{diagonal variable metric})}\label{alg:A1}
\begin{algorithmic}
   \STATE {\bfseries Initialization:} Given $\xb^0\in\mathrm{dom}(F)$, and a tolerance $\varepsilon > 0$.
   \FOR{$k = 0$ {\bfseries to} $k_{\max}$}
	\STATE 1.  Choose $\Db_k\in\Sc^n_{++}$ (e.g., using $\Db_k := L_k\mathbb{I}$, where $L_k$ is given by \eqref{eq:Barzilai_BenTal}).
	\STATE 2.  Compute the proximal-gradient search direction $\db^k$ as \eqref{eq:subprob1}.
	\STATE 3.  Compute $\beta_k := \Vert\db^k\Vert_{\Db_k}$, $r_k := M_f\Vert\db^k\Vert_2$ and $\lambda_k := \Vert\db^k\Vert_{\xb^k}$.
	\STATE 4.  If $\beta_k \leq \varepsilon$ then terminate.
	\STATE 5.  If $\beta_k^2r_k \leq (e^{r_k}-1)\lambda_k^2$, then compute $\alpha_k := \frac{1}{r_k}\ln\left(1 + \frac{\beta_k^2r_k}{\lambda_k^2}\right)$ and update $\xb^{k+1} := \xb^k + \alpha_k\db^k$. Otherwise, set $\xb^{k+1} := \xb^k$ and update $\Db_{k+1}$ from $\Db_k$.
   \ENDFOR
\end{algorithmic}
\end{algorithm}

We combine the above analysis to obtain the following proximal gradient algorithm for solving \eqref{eq:composite_min_Fx}.
The main step  in Algorithm \ref{alg:A1} is to compute the search direction $\db^k$ at Step 2, which is  equivalent to the solution of the convex subproblem \eqref{eq:subprob1}. 
The second main step is to compute $\lambda_k = \iprods{\nabla^2f(\xb^k)\db^k, \db^k}^{1/2}$. 
This quantity requires the product of Hessian $\nabla^2f(\xb^k)$ of $f$ and $\db^k$, but not the full-Hessian.
It is clear that if $\beta_k = 0$ then $\db^k = 0$ and $\xb^{k+1} \equiv \xb^k$ and we obtain the  solution of \eqref{eq:composite_min_Fx}, i.e., $\xb^k \equiv \xopt$.
The diagonal matrix $\Db_k$ can be updated as $\Db_{k+1} := c\Db_k$ for a given factor $c > 1$.

We now explain how the new theory enhances the standard backtracking linesearch approaches. 
For simplicity, let us assume $\Db_k := L_k\mathbb{I}$, where $\mathbb{I}$ is the identity matrix. 
By a careful inspection of \eqref{eq:grad_descent}, we see that $L_k = \sigma_{\max}(\nabla^2f(\xb^k))$ achieves the maximum guaranteed decrease (in the worst case sense) in the objective. 
There are many principled ways of approximating this constant based on the secant equation underlying the quasi-Newton methods. 
In Section \ref{sec:app}, we use Barzilai-BenTal's rule:
\begin{equation}\label{eq:Barzilai_BenTal}
L_k := \frac{\Vert\yb^k\Vert^2_2}{\iprods{\yb^k, \sb^k}}, ~\textrm{where}~\sb^k := \xb^k - \xb^{k\!-\!1} ~\text{and}~\yb^k := \nabla{f}(\xb^k) \!-\! \nabla{f}(\xb^{k\!-\!1}).
\end{equation}
We then deviate from the standard backtracking approaches. 
As opposed to, for instance, checking the Armijo-Goldstein condition, we use a \emph{new analytic condition} (i.e., Step 5 of Algorithm \ref{alg:A1}), which is computationally cheaper in many cases. 
Our analytic step-size then further refines the solution based on the worst-case problem structure, even if the backtracking update satisfies the Armijo-Goldstein condition.

Surprisingly, our analysis also enables us to also establish local linear convergence as described in Theorem \ref{th:convergence_of_prox_grad} under mild assumptions. 
The proof can be found in appendix. 

\vspace{-1ex}
\begin{theorem}\label{th:convergence_of_prox_grad}
Let $\set{\xb^k}_{k\geq 0}$ be a sequence generated by Algorithm \ref{alg:A1}. 
Suppose that the sub-level set $\mathcal{L}_F(F(\xb^0)) := \set{\xb \in\dom{F} : F(\xb) \leq F(\xb^0)}$ is bounded and $\nabla^2f$ is nonsingular at some $\xb\in\dom{f}$.
Suppose further that $\Db_k := L_k\mathbb{I} \succeq \tau \mathbb{I}_n$ for given $\tau > 0$.
Then,  $\set{\xb^k}$ converges to $\xopt$ the solution of \eqref{eq:composite_min_Fx}.
Moreover, if $\rho_{*} := \max\set{L_k/\sigma_{\min}^{*}-1, 1 - L_k/\sigma_{\max}^{*}} < \frac{1}{2}$  for $k$ sufficiently large then the sequence $\set{\xb^k}$ locally converges to $\xopt$ at a linear rate, where $\sigma_{\min}^*$ and $\sigma_{\max}^*$ are the smallest and the largest eigenvalues of  $\nabla^2f(\xopt)$, respectively.  
\end{theorem}
\vspace{-1ex}

\noindent\textbf{Linear convergence:}
According to Theorem \ref{th:convergence_of_prox_grad}, linear convergence is only possible when the condition number $\kappa$ of the Hessian at the true solution satisfies $\kappa=\sigma_{\max}^{*}/\sigma_{\min}^{*}<3$. 
While this seems too imposing, we claim that, for most $f$ and $g$ , this requirement is not too difficult to satisfy (see also the empirical evidence in Section  \ref{sec:app}). 
This is because the proof of Theorem \ref{th:convergence_of_prox_grad} only needs the smallest and the largest eigenvalues of  $\nabla^2f(\xopt)$, \emph{restricted} to the subspaces of the union of $\xopt - \xb^k$ for $k$ sufficiently large, to satisfy the conditions imposed by $\rho_{*}$. 
For instance,  when $g$ is based on the $\ell_1$-norm/the nuclear norm, the differences $\xopt - \xb^k$ have at most twice the sparsity/rank of $\xopt$ near convergence. 
Given such subspace restrictions, one can prove, via probabilistic assumptions on $f$ (cf., \cite{Bach2009}), that the restricted condition number is not only dramatically smaller than the full condition number $\kappa$ of the Hessian $\nabla^2f(\xopt)$, but also it can even be dimension independent with high probability.

\vspace{-2ex}
\subsection{Proximal-Newton algorithm}\label{subsec:proximal_newton}
\vspace{-1ex}
The case $\Hb_k \equiv \nabla^2f(\xb^k)$ deserves a special attention as the step-size selection rule becomes explicit and backtracking-free. 
The resulting method is a \textit{proximal-Newton} method and can be computationally attractive in certain big data problems due to its low iteration count.

The main step of the proximal-Newton algorithm is to compute the proximal-Newton search direction $\db^k$ as:
\begin{equation}\label{eq:subprob2}
\db^k := \sb^k - \xb^k,~\text{where}~\sb^k  := \mathrm{arg}\!\min_{\xb}\set{Q_{\nabla^2f(\xb^k)}(\xb,\xb^k) + g(\xb)}.
\end{equation}
Then, it updates the sequence $\set{\xb^k}$ by:
\begin{equation}\label{eq:prox_newton_scheme}
\xb^{k+1} := \xb^k + \alpha_k\db^k = (1-\alpha_k)\xb^k + \alpha_k\sb^k,
\end{equation}
where $\alpha_k \in (0, 1]$ is the step size.
If we set $\alpha_k = 1$ for all $k\geq 0$, then \eqref{eq:prox_newton_scheme} is called the full-step proximal-Newton method. Otherwise, it is a damped-step proximal-Newton method.

First, we show how to compute the step size $\alpha_k$ in the following lemma, which is a  direct consequence of Lemma \ref{le:grad_alg} by taking $\Hb_k \equiv \nabla^2f(\xb^k)$.

\vspace{-1ex}
\begin{lemma}\label{le:damped_PN_iterations}
Let $\set{\xb^k}_{k\geq 0}$ be a sequence generated by the proximal-Newton scheme \eqref{eq:prox_newton_scheme} starting from $\xb^0\in\dom{F}$.
Let $\lambda_k$ and $r_k$ be as defined by \eqref{eq:lbd_k_and_beta_k}.
If we choose the step-size $\alpha_k = r_k^{-1}\ln\left( 1 + r_k\right)$ then:
\begin{equation}\label{eq:damped_PN_descent}
F(\xb^{k+1}) \leq F(\xb^k) - r_k^{-1}\lambda_k^2\left[\left(1 + r_k^{-1}\right)\ln\left(1 + r_k\right) - 1\right].
\end{equation}
Moreover, this step-size $\alpha_k$ is optimal $($w.r.t. the worst-case performance$)$.
\end{lemma}
\vspace{-1ex}

Next, Theorem \ref{th:prox_newton_scheme} proves the local quadratic convergence of  the full-step proximal-Newton method, whose proof can be found in the appendix. 

\vspace{-1ex}
\begin{theorem}\label{th:prox_newton_scheme}
Suppose that the sequence $\set{\xb^k}_{k\geq 0}$ is generated by \eqref{eq:prox_newton_scheme} with full-step, i.e., $\alpha_k = 1$ for $k\geq 0$. If $r_k \leq \ln(4/3) \approx 0.28768207$ then it holds that:
\begin{equation}\label{eq:main_estimate}
\left({\lambda_{k+1}}/{\sqrt{\sigma^{k+1}_{\min}}}\right) \leq 2M_f\left({\lambda_k}/{\sqrt{\sigma^k_{\min}}}\right)^2,
\end{equation}
where $\sigma^k_{\min}$ is the smallest eigenvalue of $\nabla^2f(\xb^k)$.
Consequently, if we choose $\xb^0$ such that $\lambda_0 \leq \sigma_{\min}(\nabla^2f(\xb^0))\ln(4/3)$, then the sequence $\set{\lambda_k/\sqrt{\sigma^k_{\min}}}$ converges to zero at a quadratic rate.
\end{theorem}
\vspace{-1ex}

Theorem \ref{th:prox_newton_scheme} rigorously establishes where we can take full steps and still have  quadratic convergence. Based on this information, we propose the  proximal-Newton algorithm as in Algorithm \ref{alg:A3}. 

\vspace{-3ex}
\begin{algorithm}[!ht]\caption{(Prototype proximal-Newton algorithm)}\label{alg:A3}
\begin{algorithmic}
   \STATE {\bfseries Initialization:}  Given $\xb^0\in\dom{F}$ and $\sigma \in (0, \sigma_{\min}(\nabla^2f(\xb^0))\ln(4/3)]$.
   \FOR{$k = 0$ {\bfseries to} $k_{\max}$}
	\STATE 1.  Compute $\sb^k$ by\eqref{eq:subprob2}. Then, define $\db^k := \sb^k - \xb^k$ and $\lambda_k := \Vert\db^k\Vert_{\xb^k}$.
	\STATE 2.  If $\lambda_k \leq \varepsilon$, then terminate.
	\STATE 3.  If $\lambda_k > \sigma$, then compute $r_k := M_f\Vert\db^k\Vert_2$ and $\alpha_k := \frac{1}{r_k}\ln\left(1 + r_k\right)$; else $\alpha_k := 1$.
	\STATE 4.  Update $\xb^{k+1} := \xb^k + \alpha_k\db^k$.
   \ENDFOR
\end{algorithmic}
\end{algorithm}
\vspace{-3ex}
The most remarkable feature of Algorithm \ref{alg:A3} is that it does not require any globalization strategy such as backtracking line search for global convergence.

\vspace{0.5ex}
\noindent\textbf{Complexity analysis.}
First, we estimate the number of iterations needed when $\lambda_k \le \sigma$ to reach the solution $\xb^k$ such that $\frac{\lambda_k}{\sqrt{\sigma_k}} \leq \varepsilon$ for a given tolerance $\varepsilon > 0$. Based on the conclusion of Theorem \ref{th:prox_newton_scheme}, we can show that the number of iterations of Algorithm \ref{alg:A3} when $\lambda_k > \sigma$ does not exceed $k_{\max} := \left\lfloor \log_2\left(\frac{\ln(2M_f\varepsilon)}{\ln(2\sigma)}\right)\right\rfloor$.
Finally, we estimate the number of iterations needed when $\lambda_k > \sigma$. 
From Lemma \ref{le:damped_PN_iterations}, we see that for all $k\geq 0$ we have $\lambda_k \geq \sigma$ and $r_k \geq \sigma$. 
Therefore, the number of iterations is $\left\lfloor \frac{F(\xb^0) - F(\xopt)}{\psi(\sigma)} \right\rfloor$, where $\psi(\tau) := \tau\left( (1 + \tau^{-1})\ln(1 + \tau) - 1)\right) > 0$.

\vspace{-3ex}
\subsection{Proximal quasi-Newton algorithm}\label{subsec:proximal_quasi_newton}
\vspace{-1.5ex}
In many applications, estimating the Hessian $\nabla^2f(\xb^k)$ can be costly even though the Hessian is given in a closed form (cf., Section \ref{sec:app}). 
In such cases, variable metric strategies employing approximate Hessian can provide computation-accuracy tradeoffs. 
Among these approximations, applying quasi-Newton methods with BFGS updates for $\Hb_k$ would ensure its positive definiteness. 
Our analytic step-size procedures with backtracking automatically applies to the BFGS proximal-quasi Newton method, whose algorithm details and convergence analysis are omitted here.

\vspace{-2ex}
\section{Numerical experiments}\label{sec:app}
\vspace{-2ex}
We use  a variety of different real-data problems to illustrate the performance of our variable metric framework using a MATLAB implementation. 
We pick two advanced solvers for comparison: TFOCS \cite{Becker2011a} and PNOPT \cite{Lee2014}. 
TFOCS hosts accelerated first order methods. PNOPT provides a several proximal-(quasi) Newton implementations, which has been shown to be quite successful in logistic regression problems \cite{Lee2014}. Both use sophisticated backtracking linesearch enhancements. 
We benchmark all algorithms with performance profiles \cite{Dolan2002}.

 A performance profile is built based on a set $\mathcal{S}$ of $n_s$ algorithms (solvers) and a collection $\mathcal{P}$ of $n_p$ problems. We first build a profile based on computational time. We denote by
$T_{p,s} := \textit{computational time required to solve problem $p$ by solver $s$}$.
We compare the performance of algorithm $s$ on problem $p$ with the best performance of any algorithm on this problem; that is we compute the performance ratio
$r_{p,s} := \frac{T_{p,s}}{\min\{T_{p,\hat{s}} : \hat{s}\in \mathcal{S}\}}$.
Now, let $\tilde{\rho}_s(\tilde{\tau}) := \frac{1}{n_p}\mathrm{size}\left\{p\in\mathcal{P} : r_{p,s}\leq \tilde{\tau}\right\}$ for
$\tilde{\tau} \in\R_{+}$. The function $\tilde{\rho}_s:\R\to [0,1]$ is the probability for solver $s$ that a performance ratio is within a
factor $\tilde{\tau}$ of the best possible ratio. We use the term ``performance profile'' for the distribution function $\tilde{\rho}_s$ of a performance
metric.
In the following numerical examples, we plotted the performance profiles in $\log_2$-scale, i.e. $\rho_s(\tau) :=
\frac{1}{n_p}\mathrm{size}\left\{p\in\mathcal{P} : \log_2(r_{p,s})\leq \tau := \log_2\tilde{\tau}\right\}$.

\vspace{-2ex}
\subsection{Sparse logistic regression} 
\vspace{-1ex}
We consider the classical logistic regression problem of the form \cite{Yuan2011}:
\vspace{-1ex}
\begin{equation}\label{eq:classical_logistic}
\min_{\xb,\mu}\Big\{ N^{-1}\sum_{j=1}^N\log\left(1 + e^{-y_j(\iprods{\wb^{(j)}, \xb} + \mu)}\right) + \rho N^{-1/2}\norm{\xb}_1 \Big\},
\vspace{-1ex}
\end{equation} 
where $\xb\in \R^p$ is an unknown vector, $\mu$ is an unknown bias, and $y^{(j)}$ and $\wb^{j}$ are observations where $j=1,\cdots, N$. 
The logistic term in \eqref{eq:classical_logistic} is self-concordant-like with $M_f := \max \|\wb^{(j)}\|_2$ \cite{Bach2009}. 
In this case, the smooth term in \eqref{eq:classical_logistic} has Lipschitz gradient, hence several fast algorithms are applicable. 

Figure 1 illustrates the performance profiles for computational time (left) and the number of prox-operations (right) using   the $36$ medium size problems\footnote{Available at \href{http://www.csie.ntu.edu.tw/~cjlin/libsvmtools/datasets/}{http://www.csie.ntu.edu.tw/~cjlin/libsvmtools/datasets/}.}. For comparison, we use TFOCS-N07, which is Nesterov's 2007 two prox-method; and TFOCS-AT, which is Auslender and Teboulle's accelerated method, PNOPT with L-BFGS updates, and our algorithms: proximal gradient and proximal-Newton. From these performance profiles, we can observe that our proximal gradient is the best one in terms of computational time and the number of prox-operations.
In terms of time, proximal-gradient solves upto $83.3\%$ of problems with the best performance, while these numbers in TFOCS-N07 and PNOPT-LBFGS are $2.7\%$. Proximal Newton algorithm solves $11.1\%$ problems with the best performance. In prox-operations, proximal-gradient is also the best one in $75\%$ of problems.

\begin{figure}[!ht]
\vspace{-2ex}
\centerline{\includegraphics[width=0.95\textwidth]{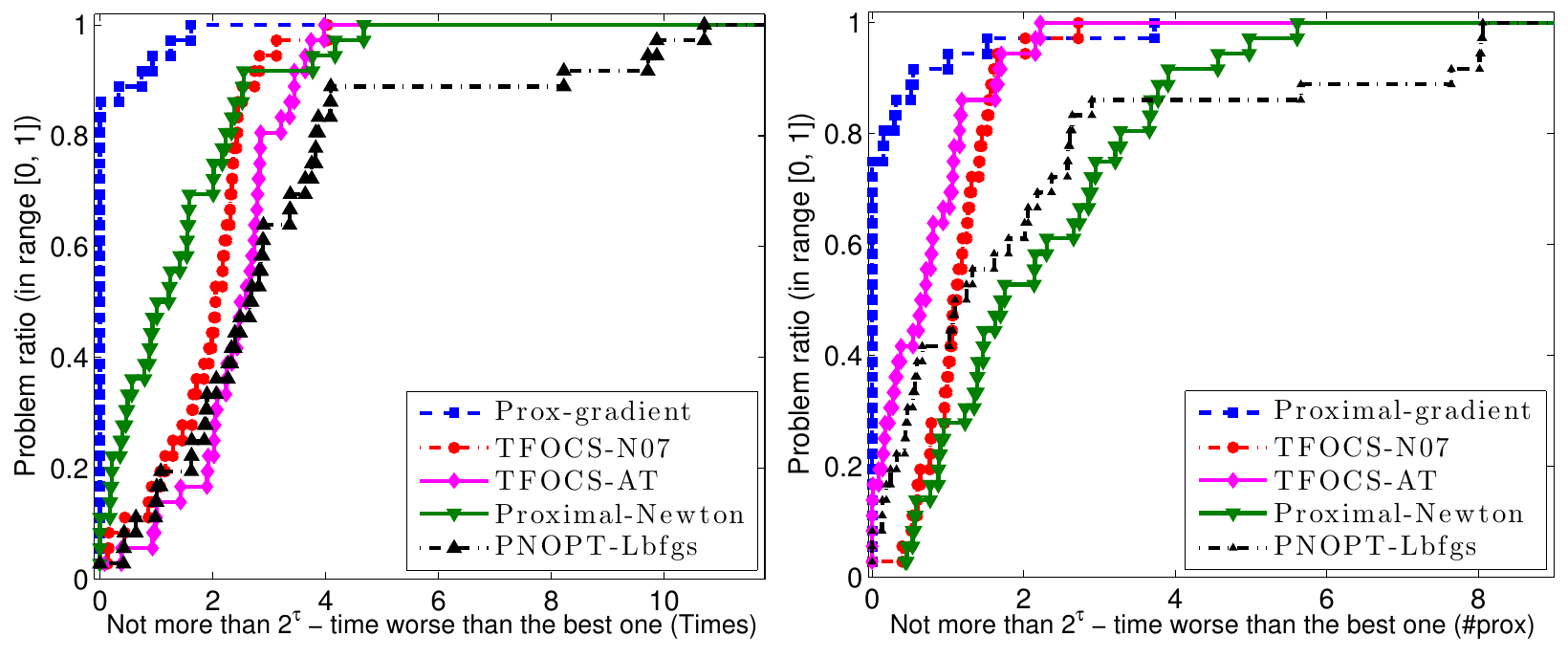}}
\vspace{-2ex}
\caption{Computational time (\textit{left})  and number of prox-operations (\textit{right})}\label{fig:logistic_profiles}
\vspace{-1ex}
\end{figure} 

\begin{figure}[!t]
\vspace{-1ex}
\centerline{\includegraphics[width=0.95\textwidth]{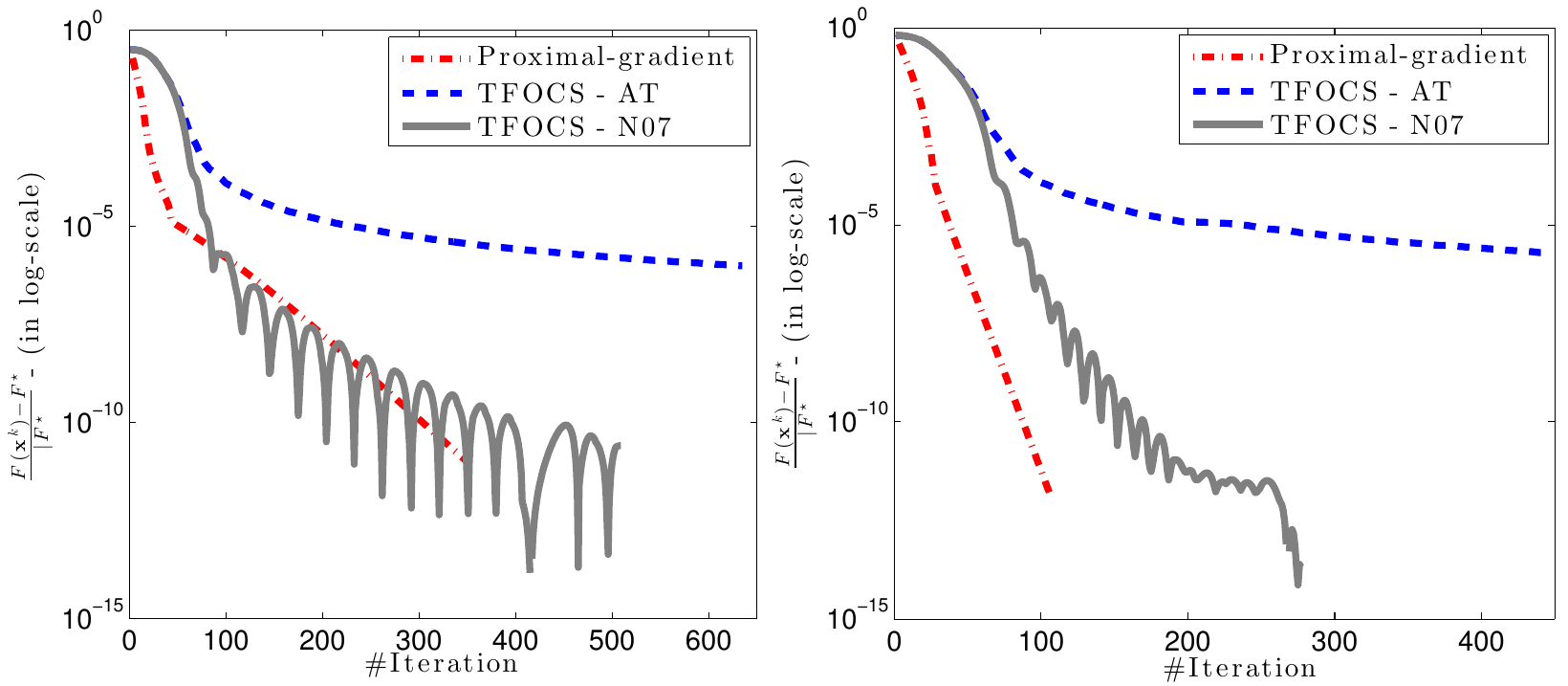}}
\vspace{-2ex}
\caption{\textit{Left}:  \texttt{rcv1\_train.binary}, and  \textit{Right:}  \texttt{real-sim}. }\label{fig:2exam_convergence}
\vspace{-3ex}
\end{figure} 

We now show an example convergence behavior of our proximal-gradient algorithm via two large-scale problems with $\rho=0.1$. The first problem is \texttt{rcv1\_train.binary} with the size $p=20242$ and $N = 47236$ and the second one is \texttt{real-sim} with the size $p=72309$ and $N = 20958$. For comparison, we use TFOCS-N07 and TFOCS-AT. For this example, PNOPT (with Newton, BFGS, and L-BFGS options) and our proximal-Newton do not scale and are omitted. 

Figure \ref{fig:2exam_convergence} shows that our simple gradient algorithm locally exhibits linear convergence whereas the fast method TFOCS-AT shows a sublinear convergence rate. The variant TFOCS-N07 is the Nesterov's dual proximal algorithm, which exhibits oscillations but performs comparable to our proximal gradient method in terms of accuracy, time, and the total number of prox operations. The computational time and the number of prox-operations in these both problems are given as follows: Proximal-gradient: (15.67s, 698), (13.71s, 152); TFOCS-AT: (20.57s, 678), (33.82s, 466); TFOCS-N07: (17.09s, 1049), (22.08s, 568), respectively. For these data sets, the relative performance of the algorithms is surprisingly consistent across various regularization parameters. 

\vspace{-3ex}
\subsection{Restricted condition number in practice} 
\vspace{-2ex}
The convergence plots in Figure  \ref{fig:2exam_convergence} indicate that the linear convergence condition in  Theorem \ref{th:convergence_of_prox_grad} may be satisfied. \emph{In fact, in all of our tests, the proximal gradient algorithm exhibits  locally linear convergence.} Hence, to see if Remark 1 is grounded in practice, we perform the following test on the \texttt{a\#a} dataset$^1$, consisting of small to medium problems. We first solve each problem with the proximal-Newton method up to 16 digits of accuracy to obtain $\xopt$, and we calculate $\nabla^2f(\xopt)$. We then run our proximal gradient algorithm until convergence, and during its linear convergence, we record $\Vert\nabla^2f(\xopt)(\xopt - \xb^k)\Vert^2/\Vert\xopt - \xb^k\Vert_2^2$, and take the ratios of the maximum and the minimum  to estimate the restricted condition number for each problem.
\begin{figure}[!ht]
\vspace{-2ex}
\centering
\centerline{\includegraphics[width=0.95\textwidth]{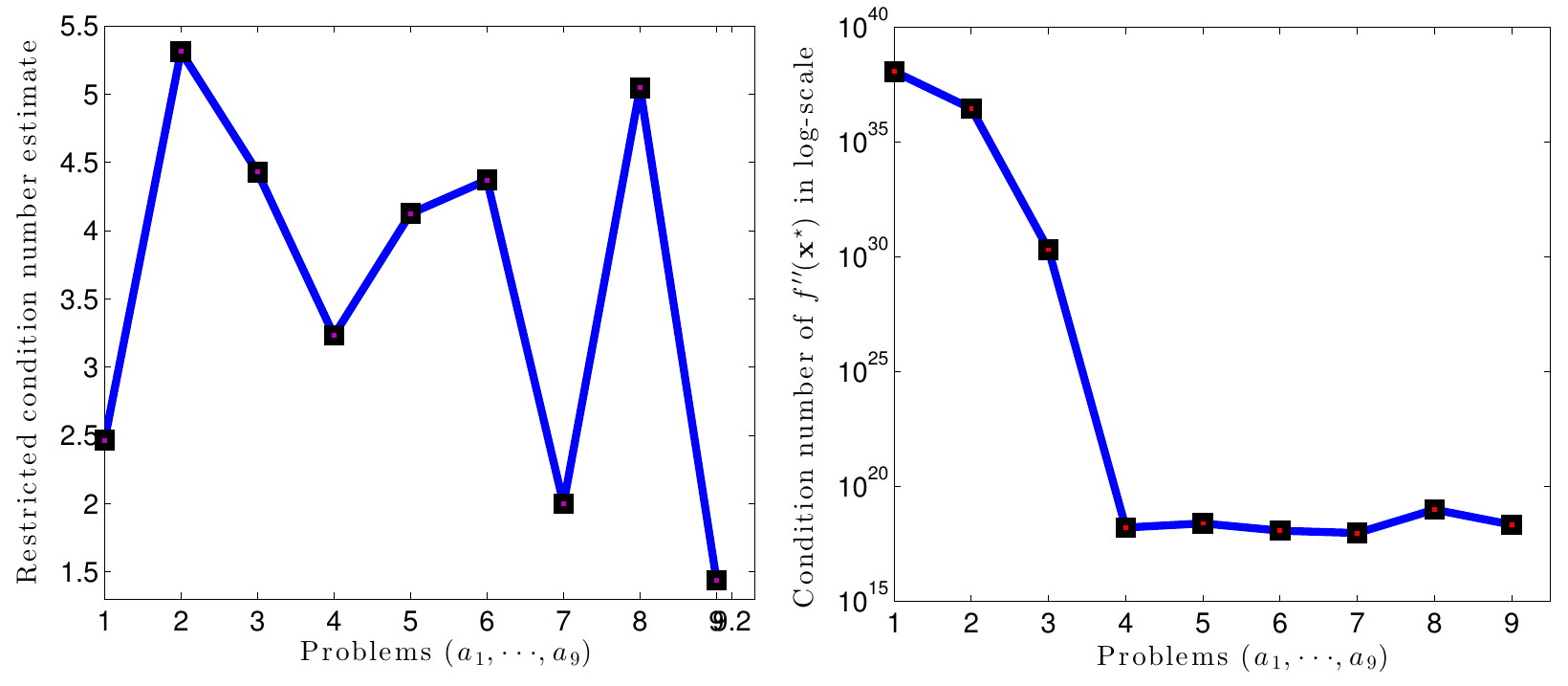}}
\vspace{-2ex}
\caption{Restricted condition number (\textit{left}), and condition number (\textit{right}) estimates}\label{fig:cond_number}
\vspace{-3.5ex}
\end{figure}

Figure \ref{fig:cond_number} illustrates that while the condition number of the Hessian $\nabla^2f(\xopt)$ can be extremely large, as the algorithm navigates to the optimal solution $\xopt$ through sparse subspaces, the restricted condition number estimates are in fact very close to $3$. 
Given that algorithm still exhibit linear convergence for the cases $\#=2,3,4,5,6,8$ (where our condition cannot be met), we believe that the tightness of our convergence condition is an artifact of our proof and may be improved. 

\vspace{-2ex}
\subsection{Sparse multinomial logistic regression}
\vspace{-1ex}
For sparse multimonomial logistic regression, the underlying problem is formulated in the form of \eqref{eq:composite_min_Fx}, which the objective function $f$ is given as:
\begin{equation}\label{eq:mn_logistic_fx}
\vspace{-1ex}
f(\mathbf{X}) := N^{-1}\sum_{j=1}^N\Big[\log\Big(1 + \sum_{i=1}^{m}e^{\iprods{\wb^{(j)}, \Xb^{(i)}}}\Big) - \sum_{i=1}^{m}\yb_i^{(j)}\iprods{\wb^{(j)}, \Xb^{(i)}}\Big].
\vspace{-0.5ex}
\end{equation}
where $\mathbf{X}$ can be considered as a matrix variable of size $m\times p$ formed from $\Xb^{(1)}, \cdots, \Xb^{(m)}$. Other vectors, $\yb^{(j)}$ and $\wb^{(j)}$ are given as input data for $j=1,\dots, N$. 
The function $f$ has closed form gradient as well as Hessian. 
However, forming a full hessian matrix $\nabla^2{f}(\xb)$ is especially costly in large scale problems when $N \gg 1$. 
In this case,  proximal-quasi-Newton methods are more suitable. 
First, we show in Lemma \ref{le:self_concordant_of_L} that $f$ satisfies Definition \ref{de:self_concordant_type}, whose proof is in the appendix. 
\vspace{-1ex}
\begin{lemma}\label{le:self_concordant_of_L}
The function $f$ defined by \eqref{eq:mn_logistic_fx} is convex and self-concordant-like in the sense of Definition \ref{de:self_concordant_type} with the parameter $M_{f} := \sqrt{6}N^{-1}\displaystyle\max_{j=1,\dots, N}\Vert\wb^{(j)}\Vert_2$.
\end{lemma}
\vspace{-1ex}
The performance profiles of $20$ small-to-medium size problems$^1$ are shown in Figure \ref{fig:logistic_profiles} in terms of computational time (left) as well as number of prox-operations (right), respectively. Both proximal-gradient method and proximal-Newton method with BFGS have good performance. They can solve unto $55\%$ and $45\%$ problems with the best time performance, respectively. These methods are also the best in terms of prox-operations ($70\%$ and $30\%$).
\begin{figure}[!t]
\vspace{-1ex}
\centerline{\includegraphics[width=0.95\textwidth]{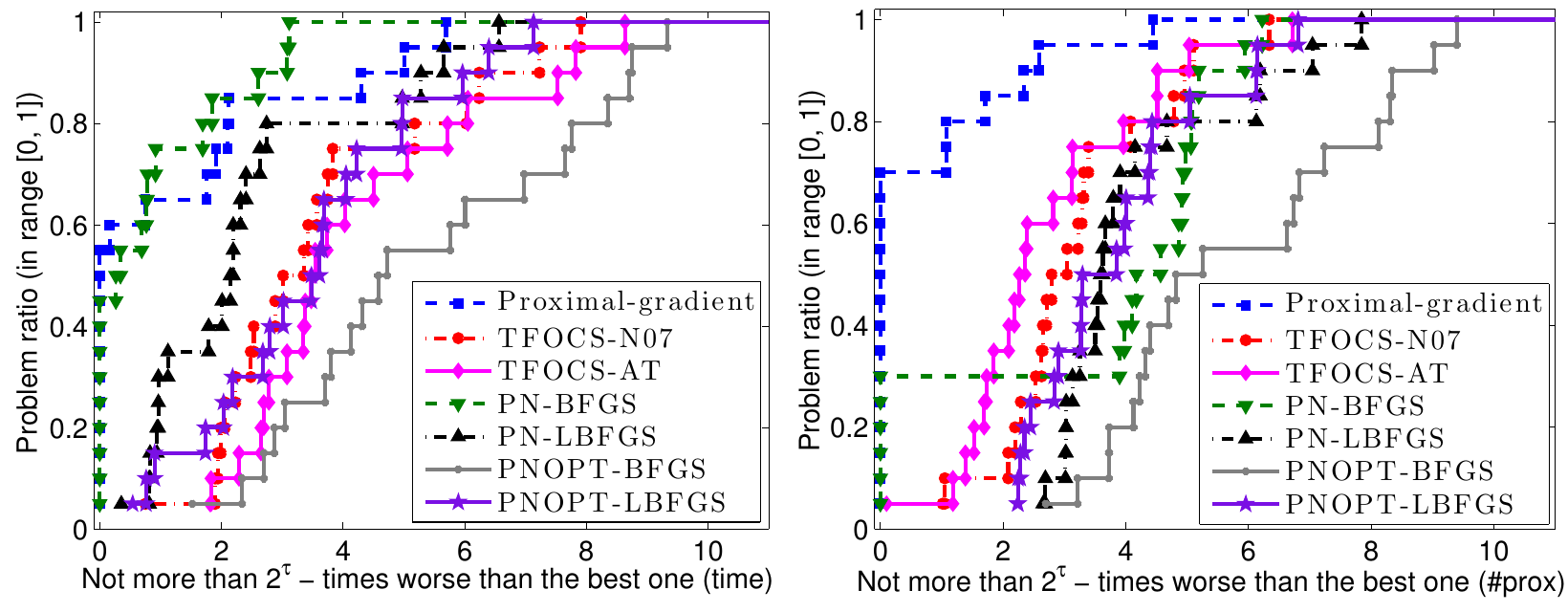}}
\vspace{-2ex}
\caption{Computational time (\textit{left}), and number of prox-operations (\textit{right})}\label{fig:logistic_profiles}
\vspace{-3ex}
\end{figure} 

\vspace{-2ex}
\subsection{A sytlized example of a non-Lipschitz gradient function for \eqref{eq:composite_min_Fx}}
\vspace{-1ex}
We consider the following convex composite minimization problem by modifying one of the canonical examples of geometric programming \cite{Boyd2004}:
\begin{equation}\label{eq:geo_prog_prob}
\vspace{-1ex}
\min_{\xb\in\Omega}\Big\{ f(\xb) := \sum_{i=1}^me^{\mathbf{a}_i^T\xb + b_i} + \mathbf{c}^T\xb\Big\} +  g(\xb),
\vspace{-0.5ex}
\end{equation}  
where $\Omega$ is a simple convex set, $\mathbf{a}_i, \mathbf{c} \in \R^n$ and $b_i\in \R$ are random, and $g$ is the $\ell_1$-norm. 
After some algebra, we can show that $f$ satisfies Definition \ref{de:self_concordant_type} with $M_f := \max\set{\norm{\mathbf{a}_i}_2 : 1\leq i\leq m}$. 
Unfortunately, $f$ does not have Lipschitz continuous gradient in $\R^n$.
 
We implement our proximal-gradient algorithm and compare it with TFOCS and PNOPT-LBFGS. However, TFOCS breaks down in running this example due to the estimation of Lipschitz constant, while PNOPT is rather slow. Several tests on synthetic data show that our algorithm outperforms PNOPT-LBFGS. As an example, we show the convergence behavior of both these methods in Figure \ref{eq:fval_vs_prox} where we plot the accuracy of the objective values w.r.t. the number of prox-operators for two cases of $\varepsilon = 10^{-6}$ and $\varepsilon = 10^{-12}$, respectively.
\begin{figure}[!ht]
\vspace{-3ex}
\centerline{\includegraphics[width=0.95\textwidth]{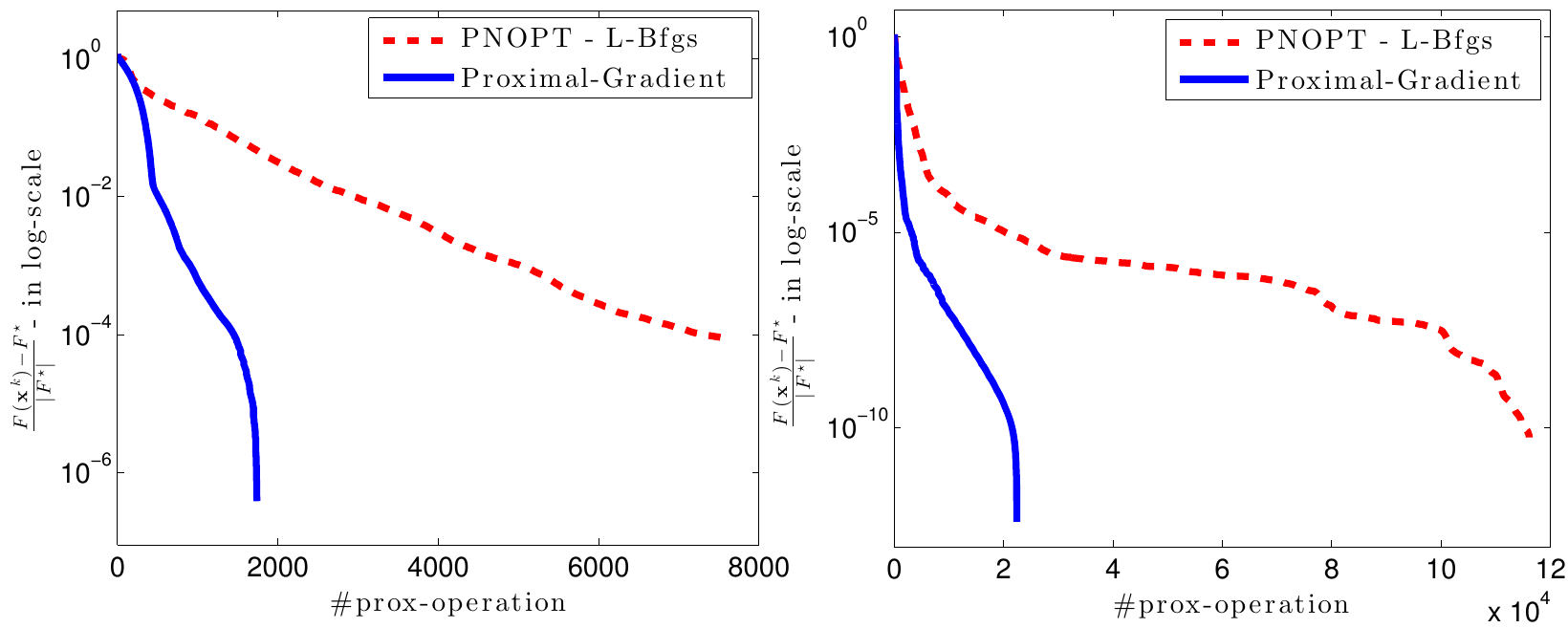}}
\vspace{-2ex}
\caption{Relative objective values w.r.t. \#prox: \textit{left}: $\varepsilon=10^{-6}$, and \textit{right}: $\varepsilon = 10^{-12}$.} \label{eq:fval_vs_prox}
\vspace{-3ex}
\end{figure} 
As we can see from this figure that our prox-gradient method requires many fewer prox-operations to achieve a very high accuracy compared to PNOPT. Moreover, our method is also $20$ to $40$ times faster than PNOPT in this numerical test.
\vspace{-2ex}
\section{Conclusions}\label{sec: conclusions}
\vspace{-2.5ex}
Convex optimization efficiency relies significantly on the structure of the objective functions. 
In this paper, we propose a variable metric method  for minimizing the sum of a self-concordant-like convex function and a proximally tractable convex function. 
Our framework is applicable in several interesting machine learning problems and do not rely on the usual Lipschitz gradient assumption on the smooth part for its convergence theory. A highlight of this work is the new analytic step-size selection procedure that enhances backtracking procedures.  Thanks to this new approach, we can prove that the basic gradient variant of our framework has improved local convergence guarantees under certain conditions while the tuning-free proximal Newton method has locally quadratic convergence. While our assumption on the restricted condition number in Theorem 1 is not deterministically verifiable \emph{a priori}, we provide empirical evidence that it can hold in many practical problems. Numerical experiments on different applications that have both self-concordant-like and Lipschitz gradient properties demonstrate that the  gradient algorithm based on the former assumption can be more efficient than the fast algorithms based on the latter assumption. As a result, we plan to look into fast versions of our gradient scheme as future work.

\bibliographystyle{plain}

\newpage
\appendix
\section{Appendix: Composite convex minimization involving self-concordant-like cost functions}
We derive some fundamental properties of self-concordant-like functions, introduce the notion of scaled proximal operators, and provide the full-proofs of the technical results in the main text.

\vspace{-2ex}
\subsection{Properties of self-concordant-like functions}
\vspace{-1ex}
We define $D^2 f(\xb) [\ub, \ub] := \left\Vert \ub \right\Vert_{\xb}^2$ and $D^3f(\xb)[\ub, \ub, \ub] := \iprods{ D^3f(\xb)[\ub]\ub, \ub}$, following the notations in \cite{Nesterov2004,Nesterov1994}.
An equivalent definition of self-concordant-like functions is provided by the following theorem \cite{Bach2009}.

\begin{theorem} \label{th:equiv_def_sclfunc}
A convex function $f \in \mathcal{C}^3: \R^n \to \R$ is $M_f$-self-concordant-like if and only if for any $\xb, \ub_1, \ub_2, \ub_3 \in \R^n$, we have:
\begin{equation}
\left\vert D^3 f(\xb) [\ub_1, \ub_2, \ub_3] \right\vert \leq M_f \left\Vert\ub_1 \right\Vert_2 \left\Vert\ub_2 \right\Vert_{\xb} \left\Vert\ub_3 \right\Vert_{\xb}. \notag
\end{equation}
\end{theorem}
Let $f: \R^n \to \R$ be an $M_f$-self-concordant-like function. Then:
\begin{itemize}
\item[$\mathrm{a)}$] The function $f_q(\xb) := \alpha + \iprods{\mathbf{a}, \xb} + (1/2)\iprods{\mathbf{A}\xb,\xb} + f(\xb)$ is $M_f$-self-concordant-like, for any $\alpha \in \R$, $\mathbf{a} \in \R^n$, and positive symmetric matrix $\mathbf{A}: \R^n \to \R^n$.
\item[$\mathrm{b)}$] The function $\alpha f$ is $M_f$-self-concordant-like for any $\alpha \geq 0$. 
\item[$\mathrm{c)}$] Let $g$ be an $M_g$-self-concordant-like funciton. Then the function $( f + g )$ is $M$-self-concordant-like, where $M := \max \left\{ M_f, M_g \right\}$.
\item[$\mathrm{d)}$] The function $f(\mathbf{A} \xb + \mathbf{b})$ is $( M_f \norm{\mathbf{A}}_2 )$-self-concordant-like, for any matrix $\mathbf{A} : \R^m \to \R^n$, $\xb \in \R^m$, and $\mathbf{b} \in \R^n$.
\end{itemize}
We will repeatedly use the following inequalities in the rest of this appendix.

\begin{theorem}\label{th:properties1} 
Let $f: \R^n \to \R$ be an $M_f$-self-concordant-like function, and define $\lambda_{\xb}(\yb) := \left\Vert \yb - \xb \right\Vert_{\xb}$ and $r_{\xb}(\yb) := M_f \left\Vert\yb - \xb \right\Vert_2$ for all $\xb, \yb \in \dom{f}$. 
We have the following inequalities for any $\xb, \yb \in \dom{f}$:
\begin{itemize}
\item[$\mathrm{a)}$] Bounds on the local norm:
\begin{equation}\label{eq:pro1}
e^{ \frac{-r_{\xb}(\yb)}{2} } \lambda_{\xb}(\yb) \leq \lambda_{\yb}(\xb) \leq e^{ \frac{r_{\xb}(\yb)}{2}} \lambda_{\xb}(\yb).
\end{equation}
\item[$\mathrm{b)}$] Bounds on the Hessian matrix: 
\begin{equation}\label{eq:pro2}
e^{ -r_{\xb}(\yb) } \nabla^2 f(\xb) \preceq \nabla^2 f(\yb) \preceq e^{ r_{\xb}(\yb) } \nabla^2 f(\xb). 
\end{equation}
\item[$\mathrm{c)}$] Bounds on the gradient vector:
\begin{equation}\label{eq:pro3}
\gamma_{*}\left( r_{\xb}(\yb) \right) \lambda_{\xb}(\yb)^2 \leq \iprods{\nabla f(\yb) - \nabla f(\xb), \mathbf{y} - \mathbf{x}} \leq \gamma \left( r_{\xb}(\yb) \right) \lambda_{\xb}(\yb)^2,
\end{equation}
where $\gamma_*( \tau ) := - \left( \frac{e^{ - \tau } - 1 }{ \tau } \right)$, and $\gamma ( \tau ) := \frac{ e^\tau - 1 }{\tau}$.
\item[$\mathrm{d)}$] Bounds on the function value:
\begin{equation}\label{eq:pro4}
\omega_{*}(r_{\xb}(\mathbf{y}))\lambda_{\xb}(\mathbf{y})^2  \leq  f(\mathbf{y}) - f(\mathbf{x}) - \nabla{f}(\mathbf{x})^T(\mathbf{x})  \leq \omega(r_{\xb}(\mathbf{y}))\lambda_{\xb}(\mathbf{y})^2,
\end{equation}
where $\omega_{*}(\tau) := \frac{e^{-\tau} + \tau - 1}{\tau^2}$ and $\omega(\tau) := \frac{e^{\tau} - \tau - 1}{\tau^2}$ are both strictly convex and increasing.
\end{itemize}
\end{theorem}

\begin{proof}
We denote $\mathbf{y}_t := \xb + t (\mathbf{y} - \mathbf{x} )$ for notational convenience, where $t \in [0, 1]$.
First we prove  \eqref{eq:pro1}. Consider the function $\phi(t) := \log\left( \iprods{\nabla^2{f}(\mathbf{y}_t )\ub,\ub}\right)$ for some $\ub \in \R^n$. 
We have, by Definition \ref{de:self_concordant_type}, that:
\begin{equation*}
\left\vert \phi'(t) \right\vert = \frac{\left\vert D^3f( \mathbf{y}_t )[\ub, \ub, \ub] \right\vert}{D^2f( \mathbf{y}_t )[\ub, \ub]} \leq M_f \left\Vert \ub \right\Vert_2.
\end{equation*}
Set $\ub = \mathbf{y} - \mathbf{x}$, and then we have $\phi ( 0 ) = \log \left( \left\Vert \mathbf{y} - \mathbf{x} \right\Vert_{\xb}^2 \right)$ and $\phi ( 1 ) = \log \left( \left\Vert \mathbf{y} - \mathbf{x} \right\Vert_{\yb}^2 \right)$.  
Integrating $\phi'(\cdot)$ over the interval $[0,1]$, we obtain:
\begin{equation}
\left\vert \log \left( \left\Vert \mathbf{y} - \mathbf{x} \right\Vert_{\yb}^2 \right) - \log \left( \left\Vert \mathbf{y} - \mathbf{x} \right\Vert_{\xb}^2 \right) \right\vert \leq M_f \left\Vert \mathbf{y} - \mathbf{x} \right\vert_2, \notag
\end{equation}
which leads to (\ref{eq:pro1}).

Next, we prove (\ref{eq:pro2}). 
Consider the function $\psi( t ) := \iprods{\nabla^2 f( \mathbf{y}_t )\ub, \ub} = \left\Vert\ub \right\Vert_{\mathbf{y}_t}^2$ for some $\ub \in \R^n$. 
We have $\psi( 0 ) = \left\Vert \ub \right\Vert_{\xb}^2$ and $\psi( 1 ) = \left\Vert\ub\right\Vert_{\yb}^2$. 
By Theorem \ref{th:equiv_def_sclfunc}, we obtain:
\begin{equation}
\left\vert \psi'(t) \right\vert = \left\vert D^3 f( \mathbf{y}_t ) [ \mathbf{y} - \mathbf{x}, \ub, \ub] \right\vert \leq M_f \left\Vert \mathbf{y} - \mathbf{x} \right\Vert_2 \psi( t ), \notag
\end{equation}
or, equivalently, 
\begin{equation}
\left\vert \frac{d\ln \psi(t)}{dt} \right\vert \leq M_f \left\Vert \mathbf{y} - \mathbf{x} \right\Vert_2. \notag
\end{equation}
We get (\ref{eq:pro2}) by integrating both sides over $[0, 1]$.

Now, we prove \eqref{eq:pro3}. By the mean-value theorem, we have:
\begin{equation}
\iprods{{f}(\mathbf{y}) - \nabla{f}(\mathbf{x}), \mathbf{y} - \mathbf{x}} = \int_0^1\iprods{\nabla^2f( \mathbf{y}_t )\mathbf{x}, \mathbf{x}}dt. \notag
\end{equation} 
Applying the right-hand side of \eqref{eq:pro2}, we obtain:
\begin{align*}
\int_0^1\iprods{\nabla^2f(\mathbf{y}_t)\mathbf{x}, \mathbf{x}}dt \leq \int_0^1 \exp \left( M_f \left\Vert \mathbf{y}_t - \xb \right\Vert_2 \right)\iprods{\nabla^2{f}(\mathbf{x})\mathbf{x}, \mathbf{x}}dt, \notag
\end{align*}
which leads to the right-hand side of \eqref{eq:pro3}. Similarly, we can prove the left-hand side of \eqref{eq:pro3}.

Finally, \eqref{eq:pro4} is a direct consequence of (\ref{eq:pro3}), since:
\begin{equation}
f(\mathbf{y}) - f(\mathbf{x}) - \iprods{\nabla{f}(\mathbf{x}), \yb - \mathbf{x}} = \int_0^1 \frac{1}{t} \iprods{\nabla{f}( \mathbf{y}_t ) - \nabla{f}(\mathbf{x}), \mathbf{y}_t - \xb}dt. \notag
\end{equation}
Hence, all the statements of Theorem \ref{th:properties1}  are proved.
\Eproof
\end{proof}

\subsection{Proof of Lemma \ref{le:existence_sol}: The existence and uniqueness of $\xopt$}

Consider the level set $\mathcal{L}_F(\mathbf{x}) := \set{\yb \in\dom{F} : F(\mathbf{y})  \leq F(\mathbf{x})}$. 
By \eqref{eq:pro4} and the convexity of $g$, for any $\yb \in \mathcal{L}_{F(\mathbf{x})}$ and $\vb \in \partial g(\mathbf{x})$, we have:
\begin{equation}
F(\xb) \geq F(\yb) \geq F(\xb) + \iprods{\nabla{f}(\mathbf{x}) + \mathbf{v}, \yb - \mathbf{x}} + \omega_{*}(r_{\xb}(\mathbf{y}))\lambda_{\xb}(\mathbf{y})^2, \notag
\end{equation}
where we use the notations $r_{\xb}(\yb)$ and $\lambda_{\xb}(\yb)$ in Theorem \ref{th:properties1}. 
Applying the Cauchy-Schwarz inequality and we obtain:
\begin{equation}
\omega_{*}(r_{\xb}(\mathbf{y}))\lambda_{\xb}(\mathbf{y}) \leq \left\Vert \nabla f(\xb) + \vb \right\Vert_{\xb}^{*}. \notag
\end{equation}
Note that $\lambda_{\xb}(\yb) \geq \sqrt{\sigma_{\min}} \left\Vert \xb - \yb \right\Vert_2$, where $\sigma_{\min}$ is the smallest eigenvalue of $\nabla^2 f(\xb)$. 
This inequality implies that:
\begin{equation}\label{eq:equation_e}
\omega_{*}(r_{\xb}(\mathbf{y}))r_{\xb}(\mathbf{y}) \leq ( M_f / \sqrt{\sigma_{\min}} ) \left\Vert \nabla f(\xb) + \vb \right\Vert_{\xb}^{*} = ( M_f / \sqrt{\sigma_{\min}} )\lambda(\xb). 
\end{equation}
The function $\psi(t) := t\omega_{*}(t)$ is increasing in $[0, +\infty)$ and its values are also in $[0, 1)$.
Moreover, $\psi(t)\to 1^{-}$ as $t\to +\infty$.
So its inverse $\psi^{-1}$ is also increasing in $[0, 1)$. 
Therefore, if $( M_f / \sqrt{\sigma_{\min}} ) \left\Vert \nabla f(\xb) + \vb \right\Vert_{\xb}^{*} < 1$, then the equation $\psi(t) - ( M_f / \sqrt{\sigma_{\min}} )\lambda(\xb) = 0$ has a unique solution $t^{\ast} > 0$.
Hence, if $r_{\xb}(\yb) \leq t^{\ast}$, then \eqref{eq:equation_e} holds.
This implies that the level set $\mathcal{L}_f(\mathbf{x})$ is bounded, and thus problem \eqref{eq:composite_min_Fx} has a solution $\xopt$. 

Let $\yb \neq \xopt$ be a point in $\dom{f}$. 
By (\ref{eq:pro4}) and the optimality condition (\ref{eq:optimality2}), for any $\vb^{\star} \in \partial{g}(\xopt)$, we have:
\begin{align}
\begin{array}{ll}
f(\yb) &\geq f(\xopt) + \iprods{\nabla{f}(\xopt), \mathbf{y} - \xopt} + \omega_{*}(r_{\xopt}(\mathbf{y}))\lambda_{\xopt}(\mathbf{y})^2,\\
g(\yb) &\geq g(\xopt) + \iprods{\vb^{\star}, \yb - \xopt} \overset{\tiny\eqref{eq:optimality2}}{=} g(\xopt) - \iprods{\nabla f(\xopt), \mathbf{y} - \xopt}. 
\end{array}
\end{align}
Summing up the two inequalities, we obtain
\begin{equation}
F(\yb) \geq F(\xopt) + \omega_{*}(r_{\xopt}(\mathbf{y}))\lambda_{\xopt}(\mathbf{y})^2.
\end{equation}
By \eqref{eq:pro2} and the non-singularity of $\nabla^2f(\xb)$ for some $\xb\in\dom{f}$, the function $f$ is strictly convex, 
and the uniqueness of $\xopt$ follows.
\eofprove

\subsection{Proof of Lemma \ref{le:grad_alg}: Step-size selection strategy}
Since $\sb^k$ is the solution of the convex subproblem (\ref{eq:subprob1}), we have $0 \in \nabla f( \xb^k ) + \mathbf{D}_k\db^k + \partial g( \sb^k )$, or, equivalently:
\begin{equation}
- ( \nabla f( \xb^k ) + \mathbf{D}_k\db^k ) \in \partial g( \sb^k ). \label{eq:optimal_condition}
\end{equation}
Using  \eqref{eq:optimal_condition} and \eqref{eq:pro4}, we can derive:
\begin{align}\label{eq:proof_a1}
f(\xb^{k+1}) &\leq f(\xb^k) + \alpha_k\iprods{\nabla{f}(\xb^k), \db^k} + \omega\left(\alpha_kr_k\right)\alpha_k^2\lambda_k^2.
\end{align}
Since $\xb^{k+1} = (1-\alpha_k)\xb^k + \alpha_k\sb^k$, it follows by the convexity of $g$ and \eqref{eq:optimal_condition} that:
\begin{align}\label{eq:proof_a2}
g(\xb^{k+1}) &\leq g(\xb^k) - \alpha_k\iprods{ \nabla f( \xb^k ) + \mathbf{D}_k\db^k, \db^k}.
\end{align}
Summing up \eqref{eq:proof_a1} and \eqref{eq:proof_a2}, we obtain the following estimate:
\begin{align}\label{eq:proof_a3}
F(\xb^{k+1}) &\leq F(\xb^k) - \psi_k( \alpha_k ), \notag
\end{align}
where $\psi_k(\tau) := \beta^2_k\tau - \lambda^2_k\omega(r_k\tau)\tau^2$. 
It is easy to check that the function $\psi_k$ is concave, and attains the maximum at $\tau^{*}_k = \frac{1}{r_k}\ln\left(1 + \frac{r_k\beta_k^2}{\lambda_k^2} \right)$ with the maximum value:
\begin{equation*}
\psi_k(\tau^{*}_k) = \frac{\beta_k^2}{r_k}\left[\left(1 + \frac{\lambda^2}{r_k\beta_k^2}\right)\ln\left(1 + \frac{\beta_k^2 r_k}{\lambda_k^2}\right) - 1\right].
\end{equation*}
Moreover, $\tau_k^{*} \leq 1$ due to the condition $\beta_k^2r_k \leq (e^{r_k}-1)\lambda_k^2$. 
By choosing $\alpha_k = \tau_k^*$, we obtain \eqref{eq:grad_descent}. 
Since $\alpha_k$ maximizes $\psi_k$, it is optimal in the sense of the worst-case performance. 
\eofprove

\subsection{Proof of Theorem \ref{th:convergence_of_prox_grad}: Local linear convergence}

\noindent\textbf{Scaled proximal operators.}
In order to prove Theorem \ref{th:convergence_of_prox_grad} and Theorem \ref{th:prox_newton_scheme}, we introduce the notion of scaled proximal operators here.
\begin{definition}
Let $\Hb\in\Sc^n_{++}$ be a positive definite matrix, and $g$ be a proper, lower semi-continuous convex function. 
We define the operator $\mathcal{P}_{\Hb}^g (\ub)$ as:
\begin{equation}
\mathcal{P}_{\Hb}^g(\ub) = \left(\Hb + \partial{g}\right)^{-1} := \mathrm{arg}\!\min_{\xb}\set{g(\mathbf{x}) + (1/2)\iprods{\Hb\xb,\xb} - \iprods{\ub, \xb}}. \notag
\end{equation}
\end{definition}
We refer to $\mathcal{P}_{\Hb}^g$ as a scaled proximity operator. 
Note that if $\Hb$ is the identity matrix, $\mathcal{P}_{\Hb}^g$ collapses to the standard proximal operator \cite{Rockafellar1970}.
For $\Hb\in\Sc^n_{++}$, we define the weighted norm of $\xb$ as $\norm{\xb}_{\Hb} := \iprods{\Hb\xb, \xb}^{1/2}$ and its dual norm of $\yb$ as $\norm{\yb}_{\Hb}^{*} := \iprods{\Hb^{-1}\yb, \yb}^{1/2}$.

\begin{lemma} \label{le:non_expansive}
Let $g$ be a proper, lower semi-continuous convex function, and let $\Hb$ be a positive definite matrix. 
The mapping $\mathcal{P}_{\Hb}^g$ is non-expansive in terms of the norm defined by $\Hb$, i.e.:
\begin{equation}\label{eq:nonexpansive_P}
\norm{\mathcal{P}_{\Hb}^g(\ub) - \mathcal{P}_{\Hb}^g(\vb)}_{\Hb} \leq \norm{\ub-\vb}_{\Hb}^{*}, ~~\forall \ub, \vb.
\end{equation} 
\end{lemma}
\begin{proof}
Let $p := \mathcal{P}_{\Hb}^g(\ub)$ and $q := \mathcal{P}_{\Hb}^g(\vb)$. 
We have $\ub - \Hb\mathbf{p} \in \partial{g}(\mathbf{p})$ and $\vb - \Hb\mathbf{q} \in \partial{g}(\mathbf{q})$. 
By the convexity of $g$, we have $\iprods{\ub - \vb - \Hb\mathbf{p} + \Hb\mathbf{q}, \mathbf{p}-\mathbf{q}} \geq 0$. 
This implies that  $\iprods{\mathbf{p} - \mathbf{q}, \ub - \vb} \geq \norm{\mathbf{p} - \mathbf{q}}_{\Hb}^2$. 
By the Cauchy-Schwarz inequality, we obtain $\norm{\ub-\vb}^{*}_{\Hb} \geq \norm{\mathbf{p} - \mathbf{q}}_{\Hb}$, which proves the theorem.
\Eproof
\end{proof}

Let us consider the distance between $\xb^{k+1}$ and $\xopt$ measured by $\left\Vert \xb^{k+1} - \xopt \right\Vert_{\xopt}$. 
By the definition of $\xb^{k+1}$, we have: 
\begin{equation} \label{eq:bound_x_k+1_x*}
\left\Vert \xb^{k+1} - \xopt \right\Vert_{\xopt} \leq ( 1 - \alpha_k ) \left\Vert \xb^k - \xopt \right\Vert_{\xopt} + \alpha_k \left\Vert \sb^k - \xopt \right\Vert_{\xopt}. 
\end{equation}
We then derive an upper bound of $\left\Vert \sb^k - \xopt \right\Vert_{\xopt} $ in terms of $\left\Vert \xb^k - \xopt \right\Vert_{\xopt}$.
We define $\Pc_{\xopt} (\ub) := \mathcal{P}_{\Hb_{\star}}^g (\ub)$ with $\Hb_{\star} := \nabla^2f( \xopt )$, $S_{\xopt}(\ub) := \nabla^2f(\xopt)\ub - \nabla{f}(\ub)$, and $e_{\xopt}(\ub, \vb) := \left[ \nabla^2f(\xopt) - \Db_k \right]\left(\vb - \ub \right)$. 
It follows from the optimality conditions \eqref{eq:optimality2} and \eqref{eq:optimal_condition} that $\sb^k = \Pc_{\xopt} \left(S_{\xopt}(\xb^k) + e_{\xopt}(\sb^k, \xb^k) \right)$ and $\xopt = \Pc_{\xopt}(S_{\xopt}(\xopt))$. 
By Lemma \ref{le:non_expansive} and the triangle inequality, we obtain:
\begin{equation} \label{eq:bound_s_k_x*}
\left\Vert \sb^k - \xopt \right\Vert_{\xopt} \leq \norm{S_{\xopt}(\xb^k) - S_{\xopt}(\xopt)}^{*}_{\xopt} + \norm{e_{\xopt}(\xb^k, \sb^k)}^{*}_{\xopt}. 
\end{equation}
Let $\tilde{r}_k := M_f \left\Vert \xb^k - \xopt \right\Vert_2$, we frist bound the term $\left\Vert S_{\xopt}\left( \xb^k \right) - S_{\xopt}\left( \xopt \right) \right\Vert_{\xopt}^*$ as follows:
\begin{equation}\label{eq:th1_est1}
\left\Vert S_{\xopt}\left( \xb^k \right) - S_{\xopt}\left( \xopt \right) \right\Vert_{\xopt}^* \leq \frac{e^{\tilde{r}_k} - \tilde{r}_k - 1}{\tilde{r}_k} \left\Vert \xb^k - \xopt \right\Vert_{\xopt}.
\end{equation}
Indeed, let us write, for notational convenience, $\mathbf{G}_k := \nabla^2 f\left( \xopt + t \left( \xb^k - \xopt \right) \right) - \nabla^2 f( \xopt )$ and $\Hb_k := \nabla^2 f( \xopt )^{-1/2}\mathbf{G}_k \nabla^2 f( \xopt )^{1/2}$. By the definition of $S_{\xopt}$, we have:
\begin{equation}\label{eq:th1_est1b}
S_{\xopt}\left( \xb^k \right) - S_{\xopt}\left( \xopt \right) = \int_0^1 \mathbf{G}_k \left( \xb^k - \xopt \right)dt.
\end{equation}
By applying the bound \eqref{eq:pro2}, we get:
\begin{equation*}
\left( \frac{1 - e^{-\tilde{r}_k}}{\tilde{r}_k} - 1 \right) \nabla^2 f\left( \xopt \right) \preceq G_k \preceq \left( \frac{e^{ \tilde{r}_k} - 1 }{\tilde{r}_k} - 1 \right) \nabla^2 f\left( \xopt \right),
\end{equation*}
which implies:
\begin{equation}\label{eq:th1_est1c}
\left\Vert \Hb_k \right\Vert \leq \max\left\{ \frac{1 - e^{-\tilde{r}_k}}{\tilde{r}_k} - 1 , \frac{e^{ \tilde{r}_k} - 1 }{\tilde{r}_k} - 1 \right\} = \frac{e^{\tilde{r}_k} - \tilde{r}_k - 1}{\tilde{r}_k}.
\end{equation}
From \eqref{eq:th1_est1b}, we can easily show that:
\begin{equation}
\left\Vert S_{\xopt}\left( \xb^k \right) - S_{\xopt}\left( \xopt \right) \right\Vert_{\xopt}^{*} \leq \left\Vert \Hb_k \right\Vert \left\Vert \xb^k - \xopt \right\Vert_{\xopt},
\end{equation}
which is combined with \eqref{eq:th1_est1c} to obtain \eqref{eq:th1_est1}.

Next, we bound the second term $\norm{e_{\xopt}(\xb^k, \sb^k)}^{*}_{\xopt}$ of \eqref{eq:bound_s_k_x*} as follows:
\begin{equation}\label{eq:th1_est2}
\norm{e_{\xopt}(\xb^k, \sb^k)}^{*}_{\xopt} \leq \rho_{*}\left(\norm{\sb^k - \xopt}_{\xopt} + \norm{\xb^k - \xopt}_{\xopt}\right),
\end{equation}
where $\rho_{*} := \max\left\{ L_k / \sigma_{\min}^*- 1, 1 - L_k / \sigma_{\max}^* \right\}$.
Indeed, let us define the matrix:
\begin{equation*}
\begin{array}{ll}
\tilde{\Hb}_{*} &:= \nabla^2f(\xopt)^{-1/2}\left[\nabla^2f(\xopt) - \Db_k\right]\nabla^2f(\xopt)^{-1/2}\\
& = \mathbb{I} - \nabla^2f(\xopt)^{-1/2}\Db_k\nabla^2f(\xopt)^{-1/2}. 
\end{array}
\end{equation*}
Then $\rho_*$ is the largest singular value of $\tilde{\Hb}_{*}$, and:
\begin{equation*}
\norm{e_{\xopt}(\xb^k, \sb^k)}^{*}_{\xopt} \leq \rho_{*} \norm{\sb^k - \xb^k}_{\xopt} \leq \rho_{*}\left(\norm{\sb^k - \xopt}_{\xopt} + \norm{\xb^k - \xopt}_{\xopt}\right), \end{equation*}
which proves \eqref{eq:th1_est2}.

Finally, suppose that $\rho_{*} \in (0, 1)$, by \eqref{eq:bound_x_k+1_x*}), \eqref{eq:bound_s_k_x*}, \eqref{eq:th1_est1}, and \eqref{eq:th1_est2}, we have the upper bound $\left\Vert x^{k + 1} - \xopt \right\Vert_{\xopt} \leq \gamma_k \left\Vert \xb^k - \xopt \right\Vert_{\xopt}$, where:
\begin{equation}
\gamma_k :=  \left\{1-\alpha_k  + \alpha_k\left[\frac{e^{\tilde{r}_k} - \tilde{r}_k - 1}{(1-\rho_{*})\tilde{r}_k}  + \frac{\rho_{*}}{1-\rho_{*}}\right]\right\}. \notag
\end{equation}
Therefore, with a proper choice of $L_k$ such that $\rho_* \in [0, 1/2)$, and for $\tilde{r}_k$ sufficiently small such that $\gamma_k < 1$, the sequence $\set{\xb^k}_{k\geq 0}$ generated by Algorithm \ref{alg:A1} converges linearly to $\xopt$.
\eofprove

\subsection{Proof of Theorem \ref{th:prox_newton_scheme}: Local quadratic convergence}
Let us define $\Pc_{\xb^k} (\ub) := P^g_{\nabla^2 f( \xb^k)} (\ub)$ with $\Hb_k := \nabla^2 f( \xb^k )$, $S_{\xb^k}(\ub) := \nabla^2 f( \xb^k )\ub - \nabla f(\ub)$, and $e_{\xb^k}(\ub, \vb) := \left[ \nabla^2 f\left( \xb^k \right) - \nabla^2 f(\ub) \right] \left(\vb - \ub\right)$. 
Since $\sb^k$ is the minimizer of (\ref{eq:subprob2}), we have $0 \in \nabla f\left( \xb^k \right) + \nabla^2 f\left( \xb^k \right) \left( \sb^k - \xb^k \right) + \partial g\left( \sb^k \right)$. Using this optimality condition and the condition $\alpha_k = 1$, we can write: 
\begin{equation}\label{eq:th2_proof1}
\begin{array}{ll}
\xb^{k+1} &= \Pc_{\xb^k} \left( S_{\xb^k}\left( \xb^k \right) + e_{\xb^k}\left( \xb^k, \sb^k \right) \right) \equiv \sb^k\\
\sb^{k+1} &= \Pc_{\xb^k} \left( S_{\xb^k}\left( \xb^{k+1} \right) + e_{\xb^k}\left( \xb^{k+1}, \sb^{k+1} \right) \right). 
\end{array}
\end{equation}
We define $\tilde{\lambda}_{k+1} := \left\Vert \sb^{k+1} - \xb^{k+1} \right\Vert_{\xb^k}$. 
It follows by Lemma \ref{le:non_expansive} and the triangle inequality that:
\begin{equation} \label{eq:bound_on_lambda_k+1}
\tilde{\lambda}_{k+1} \leq \left\Vert S_{\xb^k}\left( \xb^{k+1} \right) - S_{\xb^k}\left( \xb^k \right) \right\Vert_{\xb^k}^* + \left\Vert e_{\xb^k} \left( \sb^{k+1},\sb^{k+1} \right)  - e_{\xb^k} \left(\xb^k, \sb^k \right)\right\Vert_{\xb^k}^*
\end{equation}
We can prove, in a similar way as in the proof of \eqref{eq:th1_est1}, that the first term of \eqref{eq:bound_on_lambda_k+1} can be upper-bounded as:
\begin{equation} \label{eq:bound_on_S_k_S_k+1}
\left\Vert S_{\xb^k}\left( \xb^{k+1} \right) - S_{\xb^k}\left( \xb^k \right) \right\Vert_{\xb^k}^* \leq \frac{e^{r_k} - r_k - 1}{r_k} \left\Vert \xb^{k+1} - \xb^k \right\Vert_{\xb^k}.
\end{equation}
By using \eqref{eq:pro2} and the fact that $e_{\xb_k}(\xb^k,\sb^k) = 0$, the second term of \eqref{eq:bound_on_lambda_k+1} can be upper-bounded as:
\begin{align} \label{eq:bound_on_e_k}
\left\Vert e_{\xb^k} \left( \xb^{k+1}, \sb^{k+1} \right) \right\Vert_{\xb^k}^{*} &\leq \max \{ e^{-r_k} - 1, e^{r_k} - 1 \} \left\Vert \sb^{k+1} - \xb^{k+1} \right\Vert_{\xb^k} \nonumber\\
& = \left( e^{r_k} - 1 \right) \tilde{\lambda}_{k+1}.
\end{align}
Combining \eqref{eq:bound_on_lambda_k+1}, \eqref{eq:bound_on_S_k_S_k+1}, and \eqref{eq:bound_on_e_k} and assuming that $e^{r_k} < 2$, we have:
\begin{equation}
\tilde{\lambda}_{k+1} \leq \frac{e^{r_k} - r_k - 1}{\left( 2 - e^{r_k} \right) r_k} \lambda_k. \notag
\end{equation}
By using \eqref{eq:pro2}, we estimate $\lambda_{k+1}$ as:
\begin{equation}
\lambda_{k+1}^2  := \left\Vert \sb^{k+1} - \xb^{k+1} \right\Vert_{\xb^{k+1}}  \leq e^{r_k} \tilde{\lambda}_{k+1}^2, \notag
\end{equation}
and thus, provided that $r_k \leq \ln ( 2 )$, we obtain an upper bound for $\lambda_{k+1}$ as:
\begin{equation} \label{eq:final_bound_on_lambda_k+1}
\lambda_{k+1} \leq \frac{ e^{r_k / 2} \left( e^{r_k} - r_k - 1 \right) }{ r_k \left( 2 - e^{r_k} \right) } \lambda_k. \notag
\end{equation}
Denote by $\sigma_{\min}^k$ the smallest eigenvalue of $\nabla^2 f\left( \xb^k \right)$. We have $\left( \sigma_{\min}^{k+1} \right) ^{-1} \leq e^{r_k} \left( \sigma_{\min}^k \right) ^{-1}$ by (\ref{eq:pro2}), which, combining with (\ref{eq:final_bound_on_lambda_k+1}), gives us:
\begin{equation}\label{eq:bound_lambda_sqrt_sigma}
\frac{\lambda_{k+1}}{\sqrt{\sigma_{\min}^{k+1}}} \leq \frac{e^{r_k} \left( e^{r_k} - r_k - 1 \right)}{ r_k \sqrt{\sigma_{\min}^k} \left( 2 - e^{r_k} \right) } \lambda_k, 
\end{equation}
provided that $r_k \leq \ln(2)$. 
Finally, it is easy to check that if $r_k \leq \ln ( 4/3 ) \approx 0.28768207$, then:
\begin{equation}
\frac{e^{r_k} \left( e^{r_k} - r_k - 1 \right)}{ r_k \left( 2 - e^{r_k} \right) } \leq 2 r_k. \notag
\end{equation}
Furthermore, we note that $\lambda_k \geq \sqrt{\sigma_{\min}^k} ( r_k / M_f )$. 
Substituting these estimates into \eqref{eq:bound_lambda_sqrt_sigma} we obtain the conclusions of Theorem \ref{th:prox_newton_scheme}.
\Eproof

\subsection{Proof of Lemma \ref{le:self_concordant_of_L}: Self-concordant-like property}
The concavity of $f$ is straightforward. 
We now prove that $f$ is self-concordant-like.
Consider the function $\psi(t) := \log\left(\sum_{i=1}^me^{\mathbf{a}_it + \mu_i}\right)$, where $\mathbf{a} := (\mathbf{a}_1, \cdots, \mathbf{a}_m )$ is a fixed $m$-dimensional real vector. 
We define the polynomial $P(t; \mathbf{a}^k) := \mathbf{a}_1^ke^{\mathbf{a}_1t+\mu_1} + \cdots + \mathbf{a}_n^ke^{\mathbf{a}_nt + \mu_n}$. 
Then we have $\psi(t) = \log P(t, \mathbf{a}^0)$, and for any $k\geq 0$, $P(t; \mathbf{a}^k)_t^{\prime} := \frac{dP(t; \mathbf{a}^k)}{dt} = P(t, \mathbf{a}^{k+1})$. 

Applying the definition of $P( t, \mathbf{a}^k)$, it is straightforward to obtain the following expressions.
\begin{equation*}
\psi'(t) = \frac{P(t; \mathbf{a}^0)^{\prime}_t}{P(t; \mathbf{a}^0)} = \frac{P(t; \mathbf{a}^1)}{P(t; \mathbf{a}^0)}, ~~\psi''(t) = \frac{P(t; \mathbf{a}^2)P(t; \mathbf{a}^0) - P(t; \mathbf{a}^1)^2}{P(t; \mathbf{a}^0)^2},
\end{equation*}
and
\begin{equation}\label{eq:lm51_proof2}
\psi'''(t) = \frac{P(t; \mathbf{a}^3)P(t; \mathbf{a}^0)^2 - 3P(t; \mathbf{a}^2)P(t; \mathbf{a}^1)P(t; \mathbf{a}^0) + 2P(t; \mathbf{a}^1)^3}{P(t; \mathbf{a}^0)^3}.
\end{equation}
Let us denote $b_i := e^{\mathbf{a}_it + \mu_i}$. Then we can write $\psi''(t)$ as:
\begin{equation*}
\psi''(t) := \frac{\sum_{i < j}(\mathbf{a}_i - \mathbf{a}_j)^2b_ib_j}{(\sum_{i=1}^nb_i)^2} \geq 0,
\end{equation*}
and $\psi'''(t)$ as:
\begin{align}\label{eq:lm51_proof5}
\psi'''(t) = \frac{\sum_{i < j}(\mathbf{a}_i - \mathbf{a}_j)^2b_ib_j\left[\sum_{k=1}^n(\mathbf{a}_i + \mathbf{a}_j - 2\mathbf{a}_k)b_k\right]}{(\sum_{i=1}^nb_i)^3}.
\end{align}
We note that $\abs{\mathbf{a}_i + \mathbf{a}_j - 2\mathbf{a}_k} \leq \sqrt{6}\sqrt{\mathbf{a}_i^2 + \mathbf{a}_j^2 + \mathbf{a}_k^2} \leq \sqrt{6}\norm{\mathbf{a}}_2$ for $i, j, k = 1,\dots, m$, and $b_k \geq 0$ for all $k=1,\dots, m$. Thus we have:
\begin{align*}
\Big\vert\sum_{i=1}(\mathbf{a}_i + \mathbf{a}_j - 2\mathbf{a}_k)b_k\Big\vert \leq \sqrt{6}\norm{\mathbf{a}}_2\sum_{i=1}^mb_i.
\end{align*}
Substituting this inequality into \eqref{eq:lm51_proof5} we obtain:
\begin{equation}\label{eq:lm51_proof7}
\abs{\psi'''(t)} \leq \sqrt{6}\norm{\mathbf{a}}_2\frac{\sum_{i < j}(\mathbf{a}_i - \mathbf{a}_j)^2b_ib_j}{(\sum_{i=1}^nb_i)^2} = \sqrt{6}\norm{\mathbf{a}}_2\psi''(t).
\end{equation}
Let $\tilde{\mathbf{X}}$ be a matrix of size $(m + 1) \times p$, formed from $\tilde{\Xb}^{(1)}, \ldots, \tilde{\Xb}^{(m+1)}$. 
Now, we define the function $\tilde{f}_j ( \tilde{\mathbf{X}} ) := \log\left(\sum_{i=1}^{m+1} e^{\iprods{\wb^{(j)}, \mathbf{X}^{(i)}}}\right)$ for $j = 1, \cdots, N$, and consider $\psi_j(t) := f_j(\mathbf{X} + t\mathbf{d}) = \log\left(\sum_{i=1}^{m+1} e^{\mathbf{a}_it + \mu_i}\right)$ for given vectors $\Xb$ and $\db$, where $\mathbf{a}_i := \iprods{\mathbf{w}^{(j)}, \mathbf{d}^{(i)}}$ and $\mu_i := \iprods{\mathbf{w}^{(j)}, \Xb^{(i)}}$.
We note that:
\begin{equation} \label{eq:bound_on_a}
\mathbf{a}_2 := \left( \sum_{i=1}^{m+1}\mathbf{a}_i^2\right)^{1/2} \leq \Vert\mathbf{w}^{(j)}\Vert_2\norm{\db}_2. 
\end{equation}
By \eqref{eq:lm51_proof7} and (\ref{eq:bound_on_a}), we obtain:
\begin{equation}\label{eq:lm51_proof8}
\abs{\psi_j'''(t)} \leq  \sqrt{6}\norm{\mathbf{a}}_2\psi_j''(t) = \sqrt{6}\Vert\mathbf{w}^{(j)}\Vert_2\norm{\db}_2\psi_j''(t). \notag
\end{equation}
This inequality shows that $f_j$ is $M_j$-self-concordant-like, where $M_j = \sqrt{6}\Vert\mathbf{w}^{(j)}\Vert_2$.
By the definition of $f$ we have $f( \mathbf{X} ) = N^{-1}\left[\mathcal{A}( \mathbf{ \tilde{X} } ) - \sum_{j=1}^Nf_j( \mathbf{\tilde{X}} )\right]$ by setting $\tilde{\Xb}^{(i)} = \Xb^{(i)}$ for $i = 1, \cdots, m$, and restricting $\tilde{\Xb}^{m+1} \equiv 0$, where $\mathcal{A}$ is an affine operator. 
This implies that $f$ is $M_f$-self-concordant-like with the constant $M_f := \sqrt{6}N^{-1}\max\set{\Vert\mathbf{w}^{(j)}\Vert_2 : j=1,\cdots, N}$.
\Eproof

\end{document}